\documentclass[11pt]{amsart}
\usepackage{amssymb}
\usepackage{amsthm}
\usepackage{amsmath}
\usepackage{hyperref}
\usepackage{graphicx}
\usepackage[mathscr]{euscript}
\usepackage[all,knot,arc,cmtip, arrow]{xy}
\usepackage{tikz}
\usepackage{subfigure}

\theoremstyle{plain}
\newtheorem*{thm*}{Theorem}

\newtheorem{theorem}{Theorem}[section]
\newtheorem{lemma}[theorem]{Lemma}
\newtheorem{prop}[theorem]{Proposition}
\newtheorem{cor}[theorem]{Corollary}
\newtheorem{proposition}[theorem]{Proposition}

\theoremstyle{definition}
\newtheorem{definition}[theorem]{Definition}
\newtheorem{example}[theorem]{Example}

\theoremstyle{remark}
\newtheorem{remark}[theorem]{Remark}
\newtheorem*{claim}{Claim}
\newtheorem*{ack}{Acknowledgments}

\numberwithin{equation}{section}
\DeclareMathOperator{\cl}{\sf cl}
\newcommand{\A}{\ensuremath{\mathcal A}}
\newcommand{\sO}{\ensuremath{\mathscr O}}
\newcommand{\sR}{\ensuremath{\mathscr R}}
\newcommand{\Ve}{\ensuremath{{\sf V}}}
\newcommand{\Ed}{\ensuremath{{\sf E}}}
\newcommand{\C}{\ensuremath{\mathbb C}}
\newcommand{\R}{\ensuremath{\mathbb R}}
\newcommand{\CP}{\ensuremath{\mathbb{P}}}

\newcommand{\Z}{\ensuremath{\mathbb Z}}

\renewcommand{\bar}{\overline}
\renewcommand{\phi}{\varphi}

\DeclareMathOperator{\id}{id}
\DeclareMathOperator{\tc}{{\sf TC}}

\DeclareMathOperator{\secat}{\sf secat}
\DeclareMathOperator{\cat}{\sf cat}

\DeclareMathOperator{\hdim}{\sf {hdim}}

\begin{document}

\title{Topology of parametrised motion planning algorithms}
\author[D. Cohen]{Daniel C. Cohen}\thanks{D. Cohen was partially supported by an LSU Faculty Travel Grant}

\address{Department of Mathematics, Louisiana State University, Baton Rouge, Louisiana 70803 USA}
\email{\href{mailto:cohen@math.lsu.edu}{cohen@math.lsu.edu}}
\urladdr{\href{http://www.math.lsu.edu/~cohen/}
{www.math.lsu.edu/\char'176cohen}}

\author[M. Farber]{Michael Farber}
\thanks{M. Farber was partially supported by a grant from the Leverhulme Foundation}
\thanks{M. Farber and S. Weinberger were partially supported by an NSF - EPSRC grant}
\address{School of Mathematical Sciences, Queen Mary University of London, E1 4NS London}
\email{\href{mailto:M.Farber@qmul.ac.uk}{M.Farber@qmul.ac.uk}}

\author[S. Weinberger]{Shmuel Weinberger}
\address{Department of Mathematics, The University of Chicago, 5734 S University Ave, Chicago, IL 60637}
\email{\href{mailto:shmuel@math.uchicago.edu}{shmuel@math.uchicago.edu}}

\date{February 9, 2021}   

\begin{abstract}
We introduce and study a new concept of parameterised topological complexity, a topological invariant motivated by the motion planning problem of robotics. 
In the parametrised setting, a motion planning algorithm has high degree of universality and flexibility, it can
function under a variety of external conditions (such as positions of the obstacles etc). 
We explicitly compute the parameterised topological complexity of obstacle-avoiding collision-free motion of many 
particles (robots) in 3-dimensional space. 
Our results show that the parameterised topological complexity 
can be significantly higher than the standard (nonparametrised) 
invariant.
\end{abstract}

\keywords{parameterised topological complexity, obstacle-avoiding collision-free motion planning}

\subjclass[2010]{
55S40, 
55M30, 
55R80, 
70Q05
}

\maketitle

\section{Introduction}

To create an autonomously functioning system in robotics one designs a motion planning algorithm. 
Such an algorithm takes as input the initial and the final states of the system and produces a motion of the system from the initial to final state as output. The theory of robot motion planning algorithms is an active field of robotics, we refer the reader to the monographs \cite{Lat}, \cite{Lav} for further references. 

A topological approach to the robot motion planning problem was developed in \cite{Fa03}, \cite{Fa05}; the topological techniques explained relationships between instabilities occurring in robot motion planning algorithms and topological features of robots' configuration spaces. 

In this paper we develop a new approach to theory of motion planning algorithms. We want our algorithms to be 
{\it universal} or {\it flexible} in the sense that they should be able to function in a variety of situations, involving external conditions which are viewed as parameters and are part of the input of the algorithm. A typical situation of this kind arises when we are dealing with collision free motion of many objects (robots) moving in the 3-space avoiding a set of obstacles, and the positions of the obstacles are a priori unknown. This specific problem serves as the main motivation for us in this article and we analyse it in full detail. 

To illustrate our approach consider the following practical situation. A military commander controls a fleet of $n$ submarines in waters with $m$ mines, which are movable and are relocated every 24 hours.  Each morning 
 the commander assigns a motion for every submarine, from the current to the desired positions, such that no collisions between the submarines or between the submarines and the mines occur. A parametrised motion planning algorithm, as we discuss in this paper, will take as input the positions of the mines and the current and desired positions of the submarines and will produce as output a collision-free motion of the fleet. 
In this example the positions of the mines represent the external conditions of the system. In section \ref{sec9}
we study a mathematical version of this problem in full generality.

To model the situation mathematically we use the language of algebraic topology. We consider 
a map 
$p: E\to B$ which allows us to view $E$ as the union of a family of fibres $X=X_b=p^{-1}(b)$, parametrised by the points of the base $b\in B$. A choice of a point of the base $b\in B$ corresponds to a choice of the external conditions for the system. The parametrised topological complexity 
$\tc[p:E\to B]$ is a reflection of complexity of a {\it universal } motion planning algorithm in this setting; 
we also use the abbreviated notation $\tc_B(X_b)=\tc_B(X)$ to emphasise its relationship to the fibre. We establish several basic results describing the nature of this notion and in particular give lower and upper bounds depending on the topological spaces involved. 

Our main result is described in Section \ref{sec:config pTC} where we prove the following Theorem:
\begin{theorem}\label{main3}
The parametrised topological complexity of the problem of collision-free motion of $n$ robots in 3-space in the presence 
of $m\ge 2$ point obstacles with unknown a priori positions equals $2n+m-1$. 
\end{theorem}

This result can be compared with the known answer for the standard (i.e. nonparametrised) topological complexity 
of the problem of 
collision-free motion of $n$ robots in the presence of $m$ obstacles which equals $2n$, see \cite{FGY}. Thus we see that the parametrised topological complexity can exceed the non-parametrised one and their difference can be arbitrarily 
large, reflecting computational cost for flexible motion planning. The result shows that, unlike ordinary topological complexity, each additional obstacle imposes a cost on the motion planner. 

In this paper we consider robots and obstacles represented by single points which may appear unrealistic. 
However 
it is known that our approach is equivalent (under certain conditions) to the situation when every robot and obstacle has convex shape, see for example \cite{BBK}, \cite{Dee}.  

The main effort of the proof of Theorem \ref{main3} is in analysis of algebraic properties of the cohomology algebra of a corresponding configuration space, see Section \ref{sec:config pTC}. 

Theorem \ref{main3} is stated for the 3-dimensional case, however our methods prove a similar result in 
Euclidean space of any
odd dimension. The answer in the even-dimensional case is slightly different and its treatment requires different tools. 
We plan to describe the even-dimensional case in a separate publication. We shall also develop explicit parametrised motion planning 
algorithms  having minimal topological complexity for a variety of situations important for applications.

\section{Parametrised motion planning} \label{sec:intro}

In robotics, a motion planning algorithm takes as input pairs of admissible states of the system and generates a continuous motion of the system connecting these two states, as output. 
Let $X$ be the configuration space of the system which is a path-connected topological space.
Given a pair of states $(x_0,x_1) \in X\times X$, a motion planning algorithm produces a continuous path 
$\gamma\colon I \to X$ with $\gamma(0)=x_0$ and $\gamma(1)=x_1$, where $I=[0,1]$ is the unit interval. 
Let $X^I$ denote the space of all continuous paths in $X$ (with the compact-open topology). The map 
\begin{eqnarray}\label{for1}
\pi\colon X^I \to X\times X,\quad 
\pi(\gamma)=(\gamma(0),\gamma(1)),\end{eqnarray} 
is a fibration, with fiber $\Omega X$, the based loop space of $X$. A solution of the motion planning problem, a motion planning algorithm, is then a section of this fibration, i.e. a map $s\colon X\times X \to X^I$ satisfying $\pi\circ s=\id_{X\times X}$. The section $s$ cannot be continuous as a function of the input unless the space $X$ is contractible, see \cite{Fa03}. 

For a path-connected topological space $X$, the topological complexity $\tc(X)$ is defined to be the sectional category, or Schwarz genus, of the fibration (\ref{for1}), $\tc(X)=\secat(\pi)$. That is, $\tc(X)$ is the smallest number $k$ for which there is an open cover $X\times X=U_0\cup U_1\cup \dots \cup U_k$ and the map $\pi$ admits a continuous section $s_j\colon U_j \to X^I$ satisfying $\pi\circ s_j = \id_{U_j}$ for each $j$. Refer to the survey \cite{Fa05} and the recent volume \cite{GLV}
for detailed discussions of the 
invariant $\tc(X)$, which provides a measure of navigational complexity in $X$. 
Recent important results on $\tc(X)$ were obtained in \cite{Dr}, \cite{Gonz}. 

In this paper, we pursue a parameterised version of topological complexity, where the motion of the system is constrained by conditions imposed by points in another topological space.

To describe the setting of {\it parametrised motion planning} consider a fibration
\begin{eqnarray}\label{fib}
p\colon E\to B.\end{eqnarray}  
In general we only require (\ref{fib}) to be a Hurewicz fibration but in most applications it will be a locally trivial fibration. 
Here the space $B$ parametrises {\it the external conditions} for the system and  
for any value $b\in B$ the fibre $X_b =p^{-1}(b)$ is the space of \lq\lq real\rq\rq\,  or \lq\lq achievable\rq\rq\ configurations of the system under the external condition $b$ (Example \ref{ex1} below describes a specific situation of this kind). We assume throughout that the fibre $X_b$ is nonempty and {\it path-connected} for any $b\in B$. 
A motion planning algorithm takes as input the current and the desired states of the system 
and produces a continuous motion from the current state to the desired state. The additional conditions in the parametrised setting 
are: (a) the current and the desired states must lie in the same fibre of $p$ (i.e. they have to satisfy the same external conditions) and 
(b) the motion of the system must also be restricted to the same fibre. In practical terms this additional restriction means that the external conditions represented by the point of the base $b\in B$ will remain constant under the motion. 

In the sequel we shall denote by 
\[
F(Y,n)=\{(x_1,\dots,x_n) \in Y^n \mid x_i\neq x_j\ \text{if}\ i\neq j\}
\] 
the configuration space of $n$ distinct ordered points lying in a topological space $Y$.

\begin{example}\label{ex1}{\rm 
Consider a system represented by $n$ pairwise distinct points in the $k$-dimensional space $z_1, \dots, z_n\in \R^k$ moving in the complement of a set of obstacles $o_1, o_2, \dots, o_m\in \R^k$ which are also pairwise distinct. In practical situations one has $k=3$ or $k=2$. 
Hence, $z_i\not= z_j$,  $o_i\not=o_j$  (for $i\not=j$), and $z_i\not= o_j$. 
In the formalism of parametrised motion planning, the base space $B=F(\R^k, m)$ is the configuration space of $m$ distinct ordered points in $\R^k$ representing the obstacles, 
the total space $E=F(\R^k, n+m)$ is the configuration space of tuples $(z_1, \dots, z_n, o_1, \dots, o_m)$ and the projection $p\colon E\to B$ is given by $p(z_1, \dots, z_n, o_1, \dots, o_m) = (o_1, \dots, o_m)$. 
A motion planning algorithm produces 
for two configurations, $(z_1, \dots, z_n, o_1, \dots, o_m)$ and $(z'_1, \dots, z'_n, o_1, \dots, o_m)$,
a continuous motion  $$(z_1(t), \dots, z_n(t), o_1, \dots, o_m)\in F(\R^k, n+m),\quad \mbox{where}\, \,  t\in [0,1],$$ with $z_i(0)=z_i$ and $z_i(1)=z'_i$ for $i=1, \dots, n$. 
Observe that the positions of the obstacles are not assumed to be known in advance but are rather part of the input of the problem to be solved by the planner.}
\end{example}

%
%

The above example motivates the following general constructions and numerical invariants measuring the parametrised topological complexity; we briefly describe these notions here and will examine them in detail in Section \ref{sec:first} and in the subsequent sections. 
For a fibration $p\colon E \to B$ with fibre $X$, let $E_B^I$ denote the space of all continuous paths $\gamma\colon I\to E$ lying in a single fibre of $p$, so that the path $p\circ \gamma$ is constant. Denote by 
$E\times_B E =\{(e, e')\in E\times E\mid p(e)=p(e')\}$ the space of pairs of configurations lying in the same fibre. The map $\Pi\colon E^I_B \to E \times_B E$, $\gamma\mapsto (\gamma(0), \gamma(1))$, is a fibration, with fibre $\Omega X$, the space of based loops in $X$. Recall that we assume the fibre $X$ to be path-connected. 
We define the parameterised topological complexity of the fibration $p\colon E \to B$ to be the sectional category of the associated fibration $\Pi$, i.e. $$\tc[p\colon E \to B]:=\secat(\Pi).$$ If the fibration $p$ is clear from the context, we will sometimes use the abbreviated notation $\tc[p\colon E \to B]=\tc_B(X)$ emphasising the role of the fibre $X$. 

We shall see below that the parametrised topological complexity of a fibration 
$p\colon E 
\to B$ with fibre $X$  
can be strictly greater than the topological complexity 
$\tc(X)$, in particular this is the case in the situation of Example \ref{ex1} - collision-free motion planning with multiple obstacles.  

 As is well known, if $Y$ is a manifold without boundary of dimension at least two, the forgetful map $p\colon F(Y,n+m)\to F(Y,m)$, $p(x_1,\dots,x_{n+m})=(x_1,\dots,x_m)$, is a fibration, the Fadell-Neuwirth bundle, with fibre $F(Y\smallsetminus\{m\ \text{points}\},n)$. 
Results of Section \ref{sec9} below yield that for $k=3$ or more generally for any odd $k\ge 3$ one has
\[
\tc[p\colon F(\R^k,n+m) \to F(\R^k,m)]=2n+m-1
\]
We may compare this result with the topological complexity 
$$\tc(F(\R^k\smallsetminus\{m\ \text{points}\},n))=2n$$ of the fibre found in \cite{FGY} for $k=2$ and $k=3$ and $m\ge 2$.
\footnote{Note that \cite{FGY} deals with the non-normalised version of topological complexity which is by 1 higher than the normalised version which we use in this paper.} 
The difference between the parametrised and non-parametrised topological complexities  $\tc_B(X)-\tc(X)$ in this example 
equals $m-1$; hence this difference may be arbitrarily large.

\section{Sectional category and its generalised version}\label{sec3}
In this section, we review notions of sectional category of a fibration. 
The concept of sectional category was originally introduced by A. Schwarz \cite{Sch}, who used the term \lq\lq genus\rq\rq of a fibration.


Let $p\colon E\to B$ be a \emph{fibration}; this means that $p$ is a continuous map satisfying the homotopy lifting property with respect to any space, \cite{Spa}. 

\begin{definition}
The \emph{sectional category} 
of $p$ is the smallest integer $k\ge 0$ such that the base $B$ admits an open cover $B=U_0\cup U_1\cup \dots \cup U_k$ with the property that each set $U_i$ admits a continuous section $s_i: U_i\to E$; here $s_i$ is a continuous map with $p\circ s_i: U_i\to B$ equal to the inclusion $\id_{U_i}: U_i\to B$. 
\end{definition}

We shall denote the sectional category of $p$ by $\secat(p)$ or more informatively $\secat(p\colon E\to B)$. Note that $\secat(p\colon E\to B)=0$ if and only if $p$ admits a globally defined continuous section. 

If the base $B$ is locally contractible (which is a typical situation for this article) then any point $b\in B$ has a neighbourhood $U\subset B$ with a continuous section $U\to E$ and the sectional category $\secat(p)$ depends only 
on the global topological structure of $p$. 

For a path-connected topological space $X$ the symbol $\cat(X)$ denotes the classical Lusternik-Schnirelmann category of $X$. It can be defined as the sectional category of the Serre fibration $p_0: P_0(X) \to X$ where $P_0(X)$ denotes the space of all continuous paths $\gamma: [0,1]\to X$ with $\gamma(0) =x_0$ (the base point of $X$) and  $p_0(\gamma)=\gamma(1)$, i.e. $\cat(X) =\secat(p_0)$. It is easy to see that $\cat(X)$ can also be described as the smallest $k\ge 0$ such that $X$ admits an open cover $X=U_0\cup U_1\cup \dots\cup U_k$ with the property that each inclusion $U_i\to X$ is null-homotopic, $i=0, 1, \dots, k$.

There is also a notion of {\it generalised sectional category} of a fibration $p\colon E \to B$, which we denote by $\secat_g(p\colon E\to B)$ and abbreviate $\secat_g(p)$:


\begin{definition} \label{secatg} 
The generalised sectional category $\secat_g(p:E\to B)$ of the fibration $p\colon E \to B$
is defined as the smallest integer $k\ge 0$ such that
the base $B$ admits a partition into $k+1$ subsets
$$B=A_0\sqcup A_1\sqcup \dots \sqcup A_k, \quad A_i\cap A_j=\emptyset\ \text{for}\ i\not=j,$$
with the property that each $A_i$ admits a continuous section $s_i:A_i\to E$. 
Note that here we impose no restrictions on the nature of the sets $A_i$; in particular we do not require these sets to be open. 
\end{definition}

Clearly one has 
$\secat(p)\ge \secat_g(p)$. 
In a recent paper, Garc\'ia-Calcines \cite{GC} (see also \cite{Sri}) proved that for a fibration $p\colon E\to B$ with both spaces $E$ and $B$ being ANRs one has in fact the equality
\begin{eqnarray}\label{fib1}\label{garcia}
\secat(p\colon E\to B)= \secat_g(p\colon E\to B).
\end{eqnarray}

Recall that a metrisable topological space $X$ is {\it an absolute neighborhood retract} (ANR) if for any homeomorphism $h: X\to Y$ mapping $X$ onto a closed subset 
$h(X)$ of a metrisable topological space $Y$ there exists an open neighbourhood $U$ of $h(X)\subset Y$ which retracts onto $h(X)$.  Refer to Borsuk \cite{Bor} for a detailed discussion. Well known facts concerning ANRs include:
\begin{enumerate}
\item Any ANR is locally contractible;
\item Any polyhedron is an ANR; 
\item A metrisable topological space is an ANR if it can be represented as the union of countably many open subsets which are ANRs.
\end{enumerate}

For applications in robotics one may always restrict attention to the class of ANR spaces. 






For convenience of the reader we recall the known cohomological lower bound for the sectional category: 

\begin{prop} \label{prop:prop03}
Let $p\colon E\to B$ be a fibration and let $R$ be a ring. Suppose that a set of cohomology classes $u_0, u_1, \dots, u_k\in H^\ast(B, R)$
satisfying $p^\ast u_i=0$ for $i=0, 1, \dots, k$ has  nontrivial product
$u_0\smile u_1\smile\dots\smile u_k\not=0\in H^\ast(B, R)$. 
Then $\secat(p\colon E\to B)$ is greater than $k$. 
%
\end{prop}
\begin{proof} Assuming the contrary, $\secat(p)\le k$, 
let $B=U_0\cup U_1\cup\dots\cup U_k$ be an open cover with the property that each $U_i$ admits a homotopy section
$s_i: U_i\to E$. Then $u_i|_{U_i} = s_i^\ast p^\ast(u_i) =0$ and hence $u_i$ can be lifted to a relative cohomology class
$\tilde u_i\in H^\ast(B, U_i, R)$, i.e. $u_i = \tilde u_i|_B$. The product $\tilde u_0\smile \tilde u_1\smile \dots\smile \tilde u_k$ lies in the trivial group 
$H^\ast(B, \cup_{i=0}^k U_i, R) = H^\ast(B, B, R)=0$ and hence the product 
$$u_0\smile u_1\smile \dots\smile u_k= (\tilde u_0\smile \tilde u_1\smile \dots\smile \tilde u_k)|_B =0$$
vanishes which contradicts our assumptions.
\end{proof}

\section{The concept of pararmeterised topological complexity} \label{sec:first}

In this section, we give the general definition of parameterised topological complexity, and discuss some initial examples and results.
We repeat some of the constructions mentioned briefly at the end of section \ref{sec:intro}. Let $p\colon E \to B$ be a fibration. We wish to define a topological invariant measuring the complexity of motion planning algorithms in the family of configuration spaces parameterised by $p$. 
Recall that such an algorithm is a rule taking pairs of points $e_0,e_1\in E$ with $p(e_0)=p(e_1)=b\in B$ to a continuous path $\gamma\colon I \to X_b$ in the fibre $X_b=p^{-1}(b)$ with $\gamma(0)=e_0$ and $\gamma(1)=e_1$.
Denote 
%
 by $E_B^I$ the space of all continuous paths $\gamma: I\to E$ such that the path $p\circ \gamma $ is constant; these are paths lying in single fibre of $E$. Let $E\times_B E\subset E\times E$ be the space of pairs of configurations $(e, e')$ lying in the same fibre, i.e.
 $E\times_B E =\{(e, e')\in E\times E \mid p(e)=p(e')\}$. Sending a path to its endpoints defines a map
\begin{eqnarray} \label{fibration}
\Pi: \, E^I_B \to E\times_B E, \quad \gamma\mapsto (\gamma(0), \gamma(1)),
\end{eqnarray}
Clearly, $\Pi$ is a fibration with fibre $\Omega X$, the space of based loops in $X$, where $X$ is the fibre of $p:E\to B$. 
\begin{definition}\label{main}
The parameterised topological complexity $\tc[p\colon E \to B]$ of the fibration $p\colon E \to B$ is defined as the sectional category of $\Pi\colon E^I_B \to E\times_B E$,
\[
\tc[p\colon E \to B] := \secat(\Pi\colon E^I_B \to E\times_B E).
\]
Explicitly, $\tc[p\colon E \to B]$ is equal to the smallest integer $k\ge 0$ for which the space $E\times_B E$ admits an open cover
\begin{equation} \label{eq:cove}
E\times_B E = U_0\cup U_1\cup \dots \cup U_k, 
\end{equation}
and the map $\Pi\colon E^I_B \to E\times_B E$ admits a continuous section $s_i: U_i\to E^I_B$ for each $i=0, 1, \dots, k$. 

We occasionally use the more compact notation $\tc_B(X)$ for the parameterised topological complexity
\[
\tc_B(X)=\tc[p\colon E \to B] =
\secat(\Pi\colon E^I_B \to E\times_B E).
\]
\end{definition}

For a fibration $p\colon E \to B$ and a subset $B'\subset B$, let $E'=p^{-1}(B')$, and let $p'\colon E' \to B'$ denote the restricted fibration over $B'$. Then we obviously have
\begin{equation} \label{eq:restrict}
\tc[p\colon E \to B] \ge \tc[p'\colon E'\to B'],
\end{equation}
or in abbreviated notation, $$\tc_B(X) \ge \tc_{B'}(X)$$ for $B'\subset B$. In particular, if $B'=\{b\}$ is a single point, we obtain
\begin{equation} \label{eq:restrict1}
\tc_B(X) \ge \tc(X).
\end{equation}
In sections \ref{sec8}, \ref{sec9} of this paper we shall see examples where strict inequality holds in \eqref{eq:restrict1}. Moreover, we shall see that the difference $\tc_B(X)-\tc(X)$ can be arbitrarily large.

Next we consider a few examples where (\ref{eq:restrict1}) is satisfied as an equality. 

\begin{example}\label{ex42}
Suppose that $E=X\times B$ is the trivial fibration. In engineering terms it means that the external conditions are invariable (not changing). 
Then $E\times_B E= X\times X\times B$ and $E^I_B =X^I\times B$. In this case fibration \eqref{fibration} reduces to $p\times \id: X^I\times B\to X\times X \times B$ where $p\colon X^I \to X\times X$ is given by $p(\gamma)=(\gamma(0), \gamma(1))$.  Clearly the Schwarz genus of $p\times \id$ is equal to $\tc(X)$. We see that in this case 
$\tc_B(X) =\tc(X)$, i.e. trivial parametrisation does not add complexity.
\end{example}
\begin{prop}\label{prin}
Let $p\colon E\to B$ be a principal bundle with group $G$ (a connected topological group). 
Then
\begin{equation} \label{principal}
\tc_B(G) = \cat(G)=\tc(G).
\end{equation}
\end{prop}
Here $\cat(G)$ denotes the classical Lusternik-Schnirelmann category of $G$, see section \ref{sec3}. 
\begin{proof}
Consider the homeomorphisms 
\begin{eqnarray*}
F: E\times G \to E\times_B E, \quad F(e, g) =(e, ge),
\end{eqnarray*}
and 
$$F': E\times P_0(G) \to E_B^I, \quad F'(e, \gamma) =(t\mapsto \gamma(t)e).$$
Here $P_0(G)$ denotes the spaces of continuous paths $\gamma: I\to G$ with $\gamma(0)=1\in G$. 
Let $p_0: P_0(G)\to G$ be given by $p_0(\gamma)=\gamma(1)$. 
From the commutative diagram
$$
\xymatrix{
E\times P_0(G) \ar[r]^{F'} \ar[d]_{\id\times p_0}&  E_B^I \ar[d]^\Pi \\ 
E\times G \ar[r]^F & E\times_B E
}
$$
the sectional category of $\Pi$ is equal to the sectional category of $\id\times p_0$. The latter clearly coincides with the sectional category of 
$p_0$ which is $\cat(G)$, according to the definition, see section \ref{sec3}. This shows that $\tc_B(G)=\cat(G)$ and 
the equality (\ref{principal}) now follows from the inequalities $\tc_B(G)\ge \tc(G)\ge \cat(G)$. 
\end{proof}

\begin{example}
Consider the Hopf fibration $S^3\to S^2$ with fibre $X=S^1$. In this case $B=S^2$ and $\tc(X)=\tc(S^1)=1$. By Proposition \ref{prin}, 
$\tc_B(S^1) =1$ and we shall describe a specific motion planner.
We think of $S^3$ as being the set of pairs $(z_1, z_2)\in \C\times \C$ satisfying $|z_1|^2+|z_2|^2=1$. The group 
$S^1=\{u\in \C\mid |u|=1\}$ acts on $S^3$ by $u\cdot (z_1, z_2) = (uz_1, uz_2)$ and the base $S^2$ is the space of orbits of this action. 
We may represent the base $S^2$ as $\C\cup\{\infty\}$ and then the projection $p\colon S^3\to S^2$ is given by $p(z_1, z_2) =z_1\cdot z_2^{-1}$. 

We shall 
describe a parametrised motion planning algorithm with two open sets 
$$S^3\times_{S^2} S^3 = U_0\cup U_1.$$
Here $U_0$ will denote the set of all pairs $((z_1, z_2),(z_1', z'_2))$ satisfying $\frac{z_1}{z_2}=\frac{z'_1}{z'_2}$ and 
$(z_1, z_2)\not=(-z_1, -z_2)$. The set $U_1$ consists of all pairs $((z_1, z_2),(z_1', z'_2))$ satisfying $\frac{z_1}{z_2}=\frac{z'_1}{z'_2}$ and 
$(z_1, z_2)\not=(z_1, z_2)$. The section $s_0: U_0\to (S^3)^I_{S^2}$ can be defined by 
$$s_0((z_1, z_2),(z_1', z'_2))= (t\mapsto (e^{it\phi}z_1, e^{it\phi}z_2))$$
where $\phi\in (-\pi, \pi)$ satisfies $z'_1=e^{i\phi}z_1$ and $z'_2=e^{i\phi}z_2$. 
The section $s_1: U_1\to (S^3)^I_{S^2}$ can be defined by the formula
$$s_1((z_1, z_2),(z_1', z'_2))= (t\mapsto (e^{it\phi}z_1, e^{it\phi}z_2))$$
where $\phi\in (0, 2\pi)$ satisfies $z'_1=e^{i\phi}z_1$ and $z'_2=e^{i\phi}z_2$. 
\end{example}

\begin{prop}\label{contract}
If $p\colon E \to B$ is a fibration with path-connected fibre $X$, and $\tc_B(X)=0$, then $X$ is contractible. 
Conversely, if the fibre $X$ is contractible and $E\times_B E$ is homotopy equivalent to a CW-complex then $\tc_B(X)=0$. 
\end{prop}
\begin{proof}
The first statement follows from \eqref{eq:restrict1} and the known fact that $\tc(X)=0$ if and only if $X$ is contractible. The second statement follows by applying obstruction theory. The fibre of fibration \eqref{fibration} is the loop space $\Omega X$, which is contractible if $X$ is contractible. 
\end{proof}
%
%

\begin{remark}
The case of Proposition \ref{contract} when the fibres of $p: E\to B$ are convex sets is trivial, 
however a slightly more general situation when the fibres are star-like is already not obvious. 
\end{remark}

Our main Definition \ref{main} defines parametrised topological complexity using open covers of $E\times_B E$. We note that:

\begin{prop}\label{lm17}
If $p\colon E\to B$ is a locally trivial fibration and the spaces $E$ and $B$ are metrisable separable ANRs then in Definition \ref{main} instead of open covers one may use arbitrary covers of $E\times_B E$ or, equivalently, arbitrary partitions
$$E\times_B E = F_0\sqcup F_1\sqcup \dots \sqcup F_k, \quad F_i\cap F_j=\emptyset$$
admitting continuous sections $s_i: F_i\to E^I_B$ where $i=0, 1, \dots, k$. In other words, $$\tc[p:E\to B] = \secat_g(\Pi: E_B^I\to B\times_B E).$$
\end{prop}
\begin{proof} 
Due to the result of Garc\'ia-Calcines described earlier in section 3 (see equation (\ref{garcia})), Proposition \ref{lm17} follows once we have shown that under our assumptions the spaces 
$E^I_B$ and $E\times_B E$ are ANRs. 
 
We note that the fibre $X$ of $p\colon E\to B$ is an ANR, being a neighbourhood retract of $E$ (by \cite[Chapter IV, Theorem (3.4)]{Bor}). Secondly, $X\times X$ is an ANR, by \cite[Chapter IV, Theorem (7.2)]{Bor}. And finally, applying \cite[Chapter IV, Theorem (10.5)]{Bor}, $E\times_B E$ is an ANR as it is the total space of a locally trivial fibration $E\times_B E \to B$ with fibre 
$X\times X$ obtained by pulling back the product fibration $E\times E \to B\times B$ along the diagonal $\Delta_B: B\to B\times B$. 

Next, we note that the map $E^I_B\to B$ (which sends a path $\gamma\in E^I_B$ to the constant path 
$p\circ \gamma$ in $B$) is a locally trivial fibration with fibre $X^I$. Indeed, 
if $U\subset B$ is an open subset such that $p: E\to B$ is trivial over $U$ then the fibration $E^I_B\to B$ is also trivial over $U$. Thus 
the space $E^I_B$ is the total space of a locally trivial fibration with base $B$ and fibre $X^I$. 
We observe that $X^I$ is an ANR by  \cite[Chapter IV, Theorem (5.1)]{Bor}. 
Now, since we know that both spaces $B$ and $X^I$ are ANRs we obtain that $E^I_B$ is ANR as well, by \cite[Chapter IV, Theorem (10.5)]{Bor}.
This completes the proof. 
\end{proof}

\section{Homotopy invariance}

\begin{proposition}\label{homotopy}
If fibrations $p: E\to B$ and $p':E'\to B$ are fibrewise homotopy equivalent then $\tc[p:E\to B]=\tc[p':E'\to B]$. 
\end{proposition}

Proposition \ref{homotopy} follows from the following observation: 

\begin{proposition} Consider the commutative diagram 
\[
\xymatrix{
E' \ar[dr]_{p'}\ar[rr]^{f}  && E \ar@<1ex>[dl]^{p}\ar@<1ex>[ll]^g \\
&B   && 
}
\]
where $p, p'$ are fibrations, $p'=pf$, $p=p' g$ and $g\circ f: E'\to E'$ is fibrewise homotopic to the identity map $1_{E'}: E'\to E'$. Then 
$\tc[p': E'\to B]\le \tc[p:E\to B]$. 
\end{proposition}
\begin{proof}
Let $h_t: E'\to E'$ be a homotopy satisfying $p' \circ h_t = p'$ and $h_0=1_{E'}$ while $h_1=g\circ f$. Let $U\subset E\times_B E$ be a subset with a continuous section $s: U\to E^I_B$ of $\Pi$. Consider the set $V=(f\times f)^{-1}(U)\subset E'\times_B E'$. We want to describe a continuous section $s': V\to (E')^I_B$ of the fibration $\Pi': (E')^I_B \to E'\times_B E'$. For a pair $(a, b)\in V$ and $\tau\in [0,1]$
we set
$$
s'(a, b)(\tau)= \left\{
\begin{array}{lll}
h_{3\tau}(a), &\mbox{for} & 0\le \tau\le 1/3,\\ \\

g(s(f(a),f(b))(3\tau-1)),&\mbox{for} & 1/3\le \tau\le 2/3,\\ \\

h_{3(1-\tau)}(b), & \mbox{for} & 2/3\le \tau \le 1.
\end{array}
\right.
$$
This proves our statement. 
\end{proof}

\section{Product Inequality}
In this section we establish a product inequality for parameterised topological complexity. 

\begin{prop} \label{prop:product}
Let $p'\colon E' \to B'$ and $p''\colon E'' \to B''$ be fibrations with path-connected fibres $X'$ and $X''$ respectively. Assume that the spaces $E', B', E'', B''$ are metrisable. 
Consider the product fibration $p\colon E \to B$, with fibre $X$, where $B=B'\times B''$, $E=E'\times E''$, $X=X'\times X''$ and $p=p'\times p''$. Then, 
\[
\tc[p\colon E \to B] \le \tc[p'\colon E' \to B']  + \tc[p''\colon E'' \to B''].
\]
Equivalently, in abbreviated notation, $$\tc_{B'\times B''}(X'\times X'') \le \tc_{B'}(X')  + \tc_{B''}(X'').$$
\end{prop}
\begin{proof}
Let $\Pi\colon E^I_B \to E\times_B E$, $\Pi'\colon (E')^I_{B'} \to E'\times_{B'} E'$, $\Pi''\colon (E'')^I_{B''} \to E''\times_{B''} E''$ denote the fibrations of \eqref{fibration}. It is readily checked that $E^I_B=(E')^I_{B'} \times (E'')^I_{B''}$, 
$E\times_B E=(E'\times_{B'}E')\times(E''\times_{B''}E'')$, and that the fibration $\Pi\colon E^I_B \to E\times_B E$ is equivalent to the product fibration
\[
\Pi' \times \Pi'' \colon (E')^I_{B'} \times (E'')^I_{B''} \to (E'\times_{B'} E')\times (E''\times_{B''} E'').
\]
Since the sectional category of the product fibration is at most the sum of the sectional categories of the constituent fibrations (see \cite[Proposition 22]{Sch}), the result follows.
\end{proof}

\begin{cor}
Let $p': E'\to B$ and $p'': E''\to B$ be fibrations over base $B$ with fibres $X'$ and $X''$ respectively. 
Assume that the spaces $E', E'', B$ are metrisable. Let $p\colon E \to B$ be the fibration with fibre $X=X'\times X''$, where $E=E'\times_B E''$ and $p=p'\times_B p''$. Then,
\[
\tc[p\colon E \to B] \le \tc[p'\colon E' \to B]  + \tc[p''\colon E'' \to B].
\]
Equivalently, in abbreviated notation, $$\tc_B(X'\times X'') \le \tc_{B}(X')  + \tc_{B}(X'').$$
\end{cor}
\begin{proof}
Identify $B$ with its image under the diagonal $\Delta\colon B\to B\times B$ in the base of the fibration $p'\times p''\colon E'\times E''\to B \times B$. The result then follows from Proposition \ref{prop:product} by applying inequality \eqref{eq:restrict}.
\end{proof}

\section{Upper and lower bounds}

Let $X$ denote the fibre of a fibration $p:E\to B$. We note 
the following obvious inequality:
\begin{eqnarray}\label{upper1}
\tc_B(X) \, \le \, \cat(E\times_B E).
\end{eqnarray}
Let $\dim(Y)$ denote the covering dimension of a topological space $Y$. We shall also use the notation $\hdim (Y)$ for the minimum 
dimension $\dim Z$ where $Z$ is a space homotopy equivalent to $Y$. 
We shall refer to $\hdim(Y)$ as to {\it homotopical dimension} of $Y$. 
Using the known properties of the Lusternik-Schnirelmann category we can write
\begin{eqnarray}\label{upper2}
\tc_B(X) \le \hdim(E\times_B E),
\end{eqnarray}
as follows from (\ref{upper1}). 

\begin{prop} \label{lem:upper}
Let $p\colon E\to B$ be a locally trivial fibration with path-connected fibre $X$. Assume that the topological spaces $E$ and $B$ are metrisable. Then one has 
\[
\tc_B(X) \, = \,\tc[p\colon E\to B] \le 2 \dim (X) +\dim(B).
\]
\end{prop}
\begin{proof}
We apply the general upper bound for the sectional category \cite{Sch} in terms of the dimension of the base,
$\tc_B(X) \le \dim (E\times_B E)$. Next we note that $E\times_B E$ is the total space of a locally trivial fibration over $B$ with fibre $X\times X$ implying
$\dim(E\times_B E) \le  \dim(X\times X) +\dim B\le 2 \dim X + \dim B$, cf. \cite[Chapter 7]{E}.
\end{proof}

The upper bound of Proposition \ref{lem:upper} can be improved under connectivity assumptions on the fibre $X$ of $p\colon E\to B$:

\begin{prop} \label{lem:upper2}
Let $p\colon E\to B$ be a locally trivial fibration with fibre $X$ where the spaces $E, B, X$ are CW-complexes. 
Assume that the fibre $X$ is $r$-connected. Then 
\[
\tc_B(X)  \, < \frac{\hdim(E\times_B E) +1}{r+1}
\, \le \,  \frac{2\dim X + \dim B +1}{r+1}. 
\]
\end{prop}
\begin{proof} The fibre of the fibration $\Pi: E^I_B\to E\times_B E$ (see \eqref{fibration}) is $\Omega X$; it is $(r-1)$-connected. 
The asserted upper bound follows by applying 
\cite[Theorem 5]{Sch}. Note that the left inequality is strict and not "less or equal". 
\end{proof} 



\begin{prop} \label{cup} Let $p\colon E \to B$ be a fibration with path-connected fibre. Consider the diagonal map $\Delta\colon E \to E\times_B E$, where $\Delta(e)=(e,e)$. Then the parameterised topological complexity $\tc[p\colon E\to B]=\tc_B(X)$ is greater than or equal to the cup length of the kernel 
$\ker[\Delta^\ast: H^\ast(E\times_B E; R) \to H^\ast(E;  R)]$ where $R$ is an arbitrary coefficient ring. 
In other words, if for some cohomology classes
$u_1, \dots, u_k\in H^\ast(E\times_B E;R)$ satisfying 
$$\Delta^\ast(u_i)=0, \quad i=1, \dots, k$$
the cup-product
$$u_1\smile u_2\smile\dots \smile u_k \not=0\, \in H^\ast(E\times_B E;R)$$
is nonzero 
then $\tc_B(X)\ge k$. 
\end{prop}
\begin{proof}
Consider the commutative diagram
\[
\xymatrix{
E \ar[dr]_{\Delta}\ar[rr]^c  && E^I_B \ar[dl]^{\Pi}\\
&E\times_B E  && 
}
\]
where the map $c: E \to E^I_B$ associates with each $e\in E$ the constant path $c(e)(t)=e$. Clearly $c$ is a homotopy equivalence, 
and $\Pi\circ c =\Delta$.   Therefore,  the kernel of $\Pi^\ast: H^\ast(E\times_B E;R)\to H^\ast(E^I_B;\Pi^*R)$ coincides with 
\begin{eqnarray}\label{delta1}
\ker[\Delta^\ast: H^\ast(E\times_B E;R)\to H^\ast(E;\Delta^*R)].
\end{eqnarray}
Our claim now follows from Proposition \ref{prop:prop03}. 
\end{proof}

Our next goal is to relate the kernel (\ref{delta1}) with the ideal of the zero-divisors of the fibre $X$, i.e. with 
\begin{eqnarray}\label{delta2}
\ker[\smile: H^\ast(X;R) \otimes H^\ast(X;R) \to H^\ast(X; R)].
\end{eqnarray}
Recall the classical Leray-Hirsch theorem \cite[Theorem 4D.1]{Hatcher}. 
Let $p\colon E \to B$ be a locally trivial fibration with fibre $X$ and let
$R$ be a commutative coefficient ring (typically $\Z$ or a field). Assume that:
\begin{enumerate} \label{enum:LH}
\item[(a)] \label{itm:LHa}
For each $q\ge 0$, the cohomology $H^q(X;R)$ is a free finitely generated $R$-module.
\item[(b)] \label{itm:LHb}
There exist cohomology classes $c_j \in H^*(E;R)$ such that their restrictions $\iota^*(c_j)$ form a free basis of $H^*(X;R)$ for each fibre $X$, where $\iota\colon X \to E$ denotes the inclusion.
\end{enumerate}
The Leray-Hirsch theorem states that under these hypotheses  the cohomology of the total space $H^*(E;R)$ is a free $H^*(B;R)$-module with basis $\{c_j\}$. Explicitly, the map
\begin{equation} \label{eq:LH}
\Phi\colon H^*(B;R) \otimes_R H^*(X;R) \to H^*(E;R),\quad \sum b_i \otimes \iota^*(c_j) \mapsto
\sum p^*(b_i)\smile c_j,
\end{equation}
is an isomorphism.

%

\begin{proposition} \label{prop:zero}
If a locally trivial fibration $p: E\to B$ with fibre $X$ satisfies the Leray-Hirsch theorem and the cohomology of the base $H^\ast(B;R)$ is flat as an $R$-module then the kernel 
$\ker[\Delta^\ast: H^\ast(E\times_B E;R)\to H^\ast(E;\Delta^*R)]$
is isomorphic, as an $H^\ast(B)$-module, to 
\begin{eqnarray}\label{delta3}
H^\ast(B)\otimes_R \ker[\smile: H^\ast(X;R) \otimes H^\ast(X;R) \to H^\ast(X; R)].
\end{eqnarray}
\end{proposition}
\begin{proof} Applying the K\"unneth formula, we see that the cohomology classes 
\begin{eqnarray}\label{new1}
C_{j,j'} =f^\ast(c_j\times c_{j'})\in H^\ast(E\times_B E; R)
\end{eqnarray}
restrict to a free basis on each fibre $X\times X$ of the fibration $p': E\times_B E\to B$ where $f: E\times_B E \to E\times E$ denotes the inclusion; therefore the fibration $p'$ satisfies the assumptions of the Leray-Hirsch theorem as well. We obtain an isomorphism
$$\Phi': H^\ast(B;R) \otimes H^\ast(X\times X; R)\to H^\ast(E\times_B E;R)$$
which appears in the following commutative diagram
$$
\xymatrix{
H^\ast(B;R)\otimes_R H^\ast(X\times X;R) \ar[r]^{\hskip 0.6cm \Phi'}_{\hskip 0.6cm \simeq} \ar[d]_{1\otimes \Delta_X^\ast} & H^\ast(E\times_B E;R)\ar[d]^{\Delta^\ast}\\
H^\ast(B;R)\otimes_R H^\ast(X;R) \ar[r]_{\Phi}^{\, \, \simeq} & H^\ast(E;R).
}
$$
Here $\Delta_X: X\to X\times X$ is the diagonal.  To check commutativity one compares (\ref{eq:LH}) and (\ref{new1}). 
We obtain that the kernel of $\Delta^\ast$ coincides (after applying the isomorphisms $\Phi'$ and $\Phi$) with 
$\ker(1\otimes \Delta_X^\ast)$.  Since $H^\ast(B;R)$ is flat as an $R$-module we obtain
$\ker(1\otimes \Delta_X^\ast) =H^\ast(B;R)\otimes \ker(\smile)$. 
\end{proof}

This result may seem to suggest that under the assumptions of the Leray-Hirsch theorem the cup-length of the kernel (\ref{delta1}) equals the zero-divisors cup-length of the fibre. We shall see below that it is not the case. 
The isomorphism of Proposition \ref{prop:zero} is only additive and in reality the multiplicative structure of $\ker(\Delta^\ast)$ is {\it a deformation} of the ideal of the zero-divisors of the fibre.

\section{Motion planing for a robot in 3-space and two  obstacles with unknown in advance positions} \label{sec:small}\label{sec8}

Here we consider the situation of parametrised motion planning in 3-dimensional space with two floating obstacles, it is a special case of the situation considered in Example \ref{ex1}. 
This discussion is intended to illustrate our general Theorem \ref{thm:main}. 

The obstacles are represented by two distinct points $o_1, o_2\in \R^k$ where $k=3$. The position of the robot is represented by $z\in \R^k-\{o_1, o_2\}$. 
The relevant fibration is the Fadell-Neuwith fibration
\begin{equation} \label{eq:FNsmall}
p\colon E= F(\R^k,3) \to B= F(\R^k,2),\quad p(o_1,o_2,z)=(o_1,o_2).
\end{equation}
The fibre of this locally trivial fibration has the homotopy type of a bouquet of two spheres
\[
X=p^{-1}(o_1,o_2) \cong \{z\in \R^k \mid z\neq o_1,z\neq o_2\} \simeq S^{k-1}\vee S^{k-1}.
\] 
The base space $B=F(\R^k,2) \simeq S^{k-1}$,  
has the homotopy type of a sphere, and  the total space $E=F(\R^k,3)$ has the homotopy type of a $2(k-1)$-dimensional CW-complex. 

Since the fibre $X\simeq S^{k-1}\vee S^{k-1}$ of the fibration \eqref{eq:FNsmall} is $(k-2)$-connected, Proposition \ref{lem:upper2} shows that, for any $k\ge 2$, the parameterised topological complexity admits the upper bound
\begin{equation} \label{eq:dimsmall}
\tc[p\colon F(\R^k,3) \to F(\R^k,2)] \le 3.
\end{equation}
We show below that for $k=3$ one has  $\tc[p\colon F(\R^k,3) \to F(\R^k,2)] = 3.$
Note that the (unparameterised) topological complexity of the fibre is $\tc(X)=2$, see \cite{Fa05}. 

Below we consider integral cohomology, skipping $\Z$ in the notation. It is known (see Theorem V.4.2 from \cite{FH}) that the space $F(\R^k, 3)$ has three $(k-1)$-dimensional cohomology classes
$\omega_{12}, \, \omega_{13}, \, \omega_{23}$ which satisfy the following relations 
\begin{eqnarray}\label{braid}
\omega_{ij}^2 =0, \quad \mbox{and}\quad \omega_{13}\omega_{23} =\omega_{12}(\omega_{23}-\omega_{13})
\end{eqnarray}
and generate the cohomology ring.
Next we explain the relationship between 
Theorem V.4.2 from \cite{FH} and the Lerey-Hirsch theorem; similar arguments will be used later in this paper.
First note that the class $\omega_{ij}$ is defined as the pull-back of the fundamental class $u\in H^{k-1}(F(\R^k, 2))$ 
under the projection on the $i$-th and $j$-th particles of the configuration; here the index $i=1$ corresponds to $o_1$, the index $i=2$ corresponds to $o_2$ and the index $i=3$ corresponds to $z$. 
The classes $\omega_{13}$ and $\omega_{23}$ restrict to a set of free generators of the fibre; hence we see that fibration (\ref{eq:FNsmall}) satisfies the assumptions of the 
Leray-Hirsch theorem. The class $\omega_{12}$ generates the cohomology of the base $B$. 
By the Leray-Hirsch theorem the following classes form an additive base of the integral cohomology of $F(\R^k, 3)$ in positive degrees: 
$$\omega_{12}, \, \, \omega_{13}, \, \, \omega_{23}, \, \, \omega_{12}\omega_{13},\, \,  \omega_{12}\omega_{23}.$$ The first 3 classes have degree $(k-1)$ and the last 2 classes have degree $2(k-1)$. Relation (\ref{braid}) expresses the product of the \lq\lq fibrewise\rq\rq\,  classes 
$\omega_{13}\omega_{23}$ in terms of the generators in degree $2(k-1)$ mentioned above. Note that the products 
$\omega_{12}\omega_{13}$ and $\omega_{12}\omega_{23}$ vanish when restricted to the fibre. 

The three term relation (\ref{braid}) represents the cup-product of the total space $H^\ast(E)$ as {\it a deformation} of the tensor product algebra $H^\ast(B)\otimes H^\ast(X)$. 

Next we consider cohomology of the space $E\times_B E$ which can be identified with the configuration space of tuples $(o_1, o_2, z, z')\in (\R^k)^4$ satisfying $o_1\not=o_2,$
$z\not=o_1$, $z\not= o_2$, $z'\not=o_1$, $z'\not= o_2$. By Proposition \ref{prop:zero} the fibration $p':E\times_B E\to B$ satisfies the assumptions of the Leray-Hirsch theorem and hence the additive structure of the cohomology is given by 5 classes
\begin{eqnarray}\label{one}
\omega_{12}, \, \omega_{13}, \, \omega_{23}, \, \omega'_{13}, \, \omega'_{23}\end{eqnarray}
of degree $(k-1)$, by 8 classes of degree $2(k-1)$
\begin{eqnarray}\label{two}
\omega_{13}\omega'_{13}, \, \omega_{13}\omega'_{23}, \, \omega_{23}\omega'_{13}, \, \omega_{23}\omega'_{23}, \, \omega_{12}\omega_{13}, \, \omega_{12}\omega_{23}, \, \omega_{12}\omega'_{13}, \, \omega_{12}\omega'_{23},\end{eqnarray}
and by 4 classes of degree $3(k-1)$
\begin{eqnarray}\label{three}
\omega_{12} \omega_{13}\omega'_{13}, \, \omega_{12} \omega_{13}\omega'_{23}, \, \omega_{12} \omega_{23}\omega'_{13}, \, \omega_{12} \omega_{23}\omega'_{23}.\end{eqnarray}
Multiplicatively these classes satisfy the following relations
$$\omega_{ij}^2=0, \quad {\omega'_{ij}}^2=0$$ and two three term relations
\begin{eqnarray}\label{2braid}
\omega_{13}\omega_{23}=\omega_{12}(\omega_{23}-\omega_{13}), \quad \omega'_{13}\omega'_{23}=\omega_{12}(\omega'_{23}-\omega'_{13}).\end{eqnarray}
The first relation (\ref{2braid}) follows from the fact that the classes $\omega_{ij}$ are pull-backs of the corresponding classes in $H^\ast(E)$ under the projection on the first coordinate; similarly for the second relation in (\ref{2braid}). 

Finally we examine the map $\Delta: E\to E\times_B E$ and the homomorphism induced on cohomology. Since 
$\Delta^\ast(\omega_{ij})=\Delta^\ast(\omega'_{ij})$ we see that the classes $\omega_{13}-\omega'_{13}$ and $\omega_{23}-\omega'_{23}$ lie in the kernel of $\Delta^\ast$. 
We claim that the product
\begin{eqnarray}
(\omega_{13}-\omega'_{13})^2 (\omega_{23}-\omega'_{23})\in H^\ast(E\times_BE)
\end{eqnarray}
is nonzero. Note that here we use our assumption that $k=3$ since for $k$ even the square above would vanish, since then the class $\omega_{13}-\omega'_{13}$ has degree $k-1$ (odd) and the square of any cohomology class of odd degree is zero. 
Our claim is equivalent to the non-vanishing of the product 
\begin{eqnarray*}
\omega_{13}\omega'_{13}(\omega_{23}-\omega'_{23}) &=& \omega_{13}\omega'_{13}\omega_{23} - \omega_{13}\omega'_{13}\omega'_{23}\\
&=& \omega_{13}\omega_{23}\omega'_{13} -  \omega'_{13}\omega'_{23}\omega_{13}\\
&=& \omega_{12}(\omega_{23}-\omega_{13})\omega'_{13} - \omega_{12}(\omega'_{23} -\omega'_{13})\omega_{13}\\
&=&  \omega_{12}\omega_{23}\omega'_{13} -\omega_{12}\omega_{13}\omega'_{23} .
\end{eqnarray*}
We see that this class is the difference of two distinct generators from the list (\ref{three}) and hence it is nonzero. Applying Proposition \ref{cup} combined with the upper bound 
(\ref{eq:dimsmall}) we obtain $$\tc[p:F(\R^k, 3)\to F(\R^k, 2)]=3$$ for $k=3$.\footnote{Note that this argument applies without modifications to any odd dimension $k\ge 3$.}

By Theorem 2.2 from \cite{FGY}, $\tc(\R^3 - \{o_1, o_2\})=2$ (in normalised notations). Thus we see that the parametrised topological complexity in this example is higher than the standard (i.e. non-parametrised) one.

\section{Obstacle-avoiding collision-free motion of multiple robots in the presence of multiple obstacles with unknown in advance positions} \label{sec:config pTC}\label{sec9}

In this section we generalise the result of the previous section in several directions: firstly we allow multiple robots moving ovoiding collisions, secondly we allow an arbitrary number of obstacles. Our main focus is on the case of 3-dimensional underlying space however our results are applicable more generally to Euclidean space $\R^k$ of any dimension $k$ 
with the only restriction that $k$ must be odd. 

In the case when the dimension $k$ is even the final answer is slightly different, it requires different lower and upper bounds and will be described in a separate publication. 

Our setting is as follows: we consider $m\ge 2$ obstacles in $\R^k$ represented by pairwise distinct points $z_1, \dots, z_m$, there are also $n$ robots represented by the points $z_{m+1}, \dots, z_{m+n}\in \R^k$, these points must be pairwise distinct and distinct from the positions of the obstacles. As described in Example \ref{ex1} we have to consider the Fadell-Neuwirth fibration
\begin{eqnarray}\label{FN1}
p: F(\R^k, n+m)\to F(\R^k, m), 
\end{eqnarray}
where $ (z_1, \dots, z_m, z_{m+1}, \dots, z_{m+n})\mapsto (z_1, \dots, z_m)$ and compute its parametrised topological complexity. 
In the previous section we considered the special case $m=2$, $n=1$ and $k=3$. It will be convenient to use the notation $p:E\to B$ for (\ref{FN1}); the fibre $F(\R^k - {\mathcal O}_m, n)$ of this fibration will be denoted by $X$. Here $\mathcal O_m$ denotes a configuration of $m$ distinct points representing the obstacles. 

To explain our assumption $m\ge 2$ we note that in the case $m=1$ the base of the fibration (\ref{FN1}) is contractible and hence the fibration is trivial. 
Example \ref{ex42} shows that in this case the parametrised topological complexity coincides with 
$\tc (F(\R^k - {\mathcal O}_1, n))$ which is known \cite{FGY}. 

In this section we shall prove the following theorem which can be viewed the main result of this paper. 

\begin{theorem} \label{thm:main} Let $k\ge 3$ be odd. 
The parameterised topological complexity of the motion of $n\ge 1$ non-colliding robots in $\R^k$ in the presence of $m\ge 2$ non-colliding obstacles is equal to $2n+m-1$. 
In other words, the parameterized topological complexity of the Fadell-Neuwirth bundle 
$p\colon F(\R^k,n+m) \to F(\R^k,m)$ is
\begin{equation} \label{eq:pTCFN}
{\tc}\bigl[p\colon F(\R^k,n+m) \to F(\R^k,m)\bigr]=2n+m-1.
\end{equation}
\end{theorem}

Note that for $k=3$ the standard (i.e. nonparametrised) topological complexity of the fibre of the Fadell-Neuwirth fibration is $2n$, 
see  Theorem 5.1 of \cite{FGY}.\footnote{The arguments of the proof of Theorem 5.1 from \cite{FGY} extend with minor modifications to the case 
$k\ge 5$  odd.}
Thus we see that the parametrised topological complexity exceeds the standard one by approximately the number of obstacles. 

%
%
%
%
%
%
%
%

The rest of this section is devoted to the proof of Theorem \ref{thm:main}.

First, we apply Proposition \ref{lem:upper2} to get an upper bound. We note that 
the fibre $X$ 
of the Fadell-Neuwirth fibration (\ref{FN1}) is $(k-2)$-connected and  $\hdim X = (k-1)n$, while the homotopical dimension of the base is $\hdim B= (k-1)(m-1)$.
Hence the space $E\times_B E$ has homotopical dimension $(k-1)(2n+m-1)$. 
Consequently, from the upper bound estimate of Proposition \ref{lem:upper2}, we obtain ${\tc}\bigl[p: E\to B\bigr]  < 2n+m-1+\frac{1}{k-1}$, that is
\begin{equation} \label{eq:rough upper}
\begin{aligned}
{\tc}\bigl[p\colon F(\R^k,n+m) \to F(\R^k,m)\bigr]  \le 2n+m-1.
\end{aligned}
\end{equation}
This gives an upper bound in \eqref{eq:pTCFN}.

The proof of Theorem \ref{thm:main} will use the structure of the integral cohomology ring 
$H^*(E\times_B E)$, which we describe next.

\begin{prop} \label{prop:HEBE}
The integral cohomology ring $H^*(E\times_B E)$ contains degree $k-1$ classes 
$$\omega_{ij}, \, \, \omega'_{ij}\quad \mbox{where} \quad 1\le i< j \le n+m,$$ which satisfy the following relations
\begin{eqnarray*}
\begin{array}{llll}
&\omega_{ij}^2 =(\omega'_{ij})^2=0 \quad &\text{for all}\quad & i<j,\\ \\
&\omega_{ir}\omega_{jr}= \omega_{ij}(\omega_{jr}-\omega_{ir} )    \quad &\text{for all}\quad &  i<j <r,\\ \\
&\omega'_{ir}\omega'_{jr}= \omega'_{ij}(\omega'_{jr}-\omega'_{ir} ) \quad  &\text{for all}\quad & i<j <r,\\ \\
&\omega_{ij}=\omega'_{ij}\quad &\text{for all}\quad & 1\le i<j \le m.
\end{array}
\end{eqnarray*}
\end{prop}
\begin{proof} A point of the space $E\times_B E$ can be represented by an $(m+2n)$-tuple of the form 
$$(z_1, \dots, z_m, z_{m+1}, \dots, z_{m+n}, z'_{m+1}, \dots, z'_{m+n}) \in (\R^k)^{m+2n}.
$$ where the first $m$ points represent $m$ obstacles, and the tuples 
$(z_{m+1}, \dots, z_{m+n})$ and $(z'_{m+1}, \dots, z'_{m+n})$ 
represent the initial and the final configurations of the robots. Clearly, we require that for any $i<j$ 
\begin{eqnarray}\label{not=}
z_i\not= z_j \quad \mbox{and}\quad z'_i\not=z'_j,\end{eqnarray} 
and besides
\begin{eqnarray}\label{not=2}
z_i\not= z'_j\quad \mbox{for} \quad i\le m; 
\end{eqnarray}
note that index $j$ in (\ref{not=2}) automatically satisfies $j>m$. 
Conditions (\ref{not=}) and (\ref{not=2}) guarantee that no collisions between robots and between obstacles happen. 

For any pair of indexes $i<j$ satisfying either (\ref{not=}) or (\ref{not=2}) we may view the corresponding pair of points as an element of the configuration space $F(\R^k, 2)$ of pairs of distinct points in $\R^k$. The space $F(\R^k, 2)$
is homotopy equivalent to the sphere $S^{k-1}$ and has therefore a fundamental class in $H^{k-1}(F(\R^k, 2))$. 
Using this, for $i<j$ one defines the class $\omega_{ij}\in H^{k-1}(E\times_B E)$ as the pull-back of the fundamental class
under the map $E\times_B E\to F(\R^k, 2)$ projecting configurations in $E\times_B E$ into the pairs $(z_i, z_j)$. 
Similarly, 
for $i\le m$ and $j>m$ one defines the class 
 $\omega'_{ij}\in H^{k-1}(E\times_B E)$ as the pull-back of the fundamental class under the map 
 $E\times_BE\to F(\R^k, 2)$ projecting $E\times_B E$ onto the pairs $(z_i, z'_j)$. Besides, for $m<i<j\le m+n$ we define $\omega'_{ij}$ as above by using the projection  $E\times_BE\to F(\R^k, 2)$ on the pair $(z'_i, z'_j)$. 
 Finally we formally define $\omega'_{ij}$ for $i<j\le m$ by setting
 $\omega'_{ij}=\omega_{ij}$. 
 
 All the relations mentioned in the statement of Proposition \ref{prop:HEBE} are well known 
to hold in cohomology of configuration spaces, see \cite[Chapter V]{FH}. Since our classes are pull-backs of cohomology classes originating from configuration spaces these relations hold as well. 
\end{proof}

Next we introduce the following notations. 
We shall consider sequences $I=(i_1, i_2, \dots, i_p)$ and $J=(j_1, j_2, \dots, j_p)$ 
consisting of elements of the index set $$\{1, 2, \dots, m+n\}$$ such that $i_s<j_s$ for all $s=1, \dots, p$; we shall express this by $I<J$ for brevity. 
For each such 
pair $I<J$ we define a cohomology class
$$\omega_{IJ} =\omega_{i_1 j_1}\omega_{i_2 j_2}\dots \omega_{i_p j_p}\in H^{(k-1)p}(E\times_B E)$$
as the cup-product of the classes $\omega_{ij}$ defined above. 
The classes 
$$\omega'_{IJ}\in H^{(k-1)p}(E\times_B E)$$ are defined similarly where instead of the classes $\omega_{ij}$ one takes the classes $\omega'_{ij}$. The case $p=0$ is also allowed; in this case we define $\omega_{IJ}=1=\omega'_{IJ}$. 

A sequence $J=(j_1, j_2, \dots, j_p)$ is said to be {\it increasing} if 
$j_1<j_2<\dots<j_p$.

\begin{proposition}\label{prop:63} A free basis of the abelian group $H^\ast(E\times_B E)$  is given by the set of cohomology classes of the form 
\begin{eqnarray}
\omega_{I_1J_1}\omega_{I_2J_2}\omega'_{I_3J_3},
\end{eqnarray}
where $I_1<J_1$, \, $I_2<J_2$, \, $I_3<J_3$, the sequences $J_1, J_2, J_3$ are increasing and $J_1$ takes values in $\{1, \dots, m\}$, while $J_2$ and $J_3$ take values in $\{m+1, \dots, m+n\}$. 
\end{proposition}

\begin{proof}
We want to apply the Leray-Hirsch theorem to the fibration $E\times_B E\to B$. Recall that $E=F(\R^k, n+m)$ and 
$B=F(\R^k, m)$ and the fibre of this fibration is $X\times X$ where $X=F(\R^k-\mathcal O_m, n)$.
The classes 
$\omega_{ij}$ with $i<j\le m$ originate from the base of this fibration. Moreover, it is well known that the cohomology 
of the base $H^\ast(B)=H^\ast(F(\R^k, m))$ has a free additive basis consisting of the classes $\omega_{IJ}$ where 
$I<J$ run over all sequences of elements of the set $\{1, 2, \dots, m\}$ such that the sequence $J=(j_1, j_2, \dots, j_p)$ is {\it increasing}.
 Clearly here $p$ must be at most $m-1$. 

Next consider the classes $\omega_{I_2J_2}$ and $\omega'_{I_3J_3}$ where $I_2<J_2$, \,  $I_3<J_3$
and the increasing sequences $J_2$ and $J_3$ consist of elements of the set
$\{m+1, \dots, m+n\}$. Using the known results about the cohomology algebras of configuration spaces 
(see \cite[Chapter V, Theorems 4.2 and 4.3]{FH}) as well as the K\"unneth theorem we see that the restrictions 
 of the family of classes
$\omega_{I_2J_2}\omega'_{I_3J_3}$ onto the fibre form a free basis in the cohomology of the fibre $H^\ast(X\times X)$. 
 Hence, applying the Leray-Hirsch theorem \cite{Hatcher}, we obtain that a free basis of the cohomology $H^\ast(E\times _BE)$ is given by the set of classes of the form 
\begin{eqnarray}\label{generator}
\omega_{I_1J_1}\omega_{I_2J_2}\omega'_{I_3J_3}
\end{eqnarray}
where $I_1<J_1$, $I_2<J_2$, $I_3<J_3$, the sequences $J_1, J_2, J_3$ are increasing and $J_1$ takes values in $\{1, \dots, m\}$ while $J_2$ and $J_3$ take values in $\{m+1, \dots, m+n\}$. 
\end{proof}

Using Proposition \ref{prop:63} one may show that the classes $\omega_{ij}, \omega'_{ij}$ are multiplicative generators 
of the cohomology ring $H^\ast(E\times_B E)$ and the relations mentioned in Proposition \ref{prop:HEBE} form a defining set of relation. However we shall not need this statement in this paper. 

Note that the maximal degree of the class of the form (\ref{generator}) is $(2n +m-1)(k-1)$ and the top degree is achieved by 
taking
$J_1=(2, 3, \dots, m)$ and $J_2=J_3 =(m+1, \dots, m+n)$.

Next we consider the diagonal map $\Delta: E\to E\times_B E$. 

\begin{proposition}\label{prop:64} For any $i<j$ the class 
\begin{eqnarray}\label{diff}
\omega_{ij}-\omega'_{ij}
\end{eqnarray} lies in the kernel of the homomorphism 
$\Delta^\ast: H^\ast(E\times_B E) \to H^\ast(E)$ induced by the 
diagonal map $\Delta: E\to E\times_B E$. 
\end{proposition} 
\begin{proof}
The points of the space $E=F(\R^k, {n+m})$ are represented by configurations $(z_1, z_2, \dots, z_{n+m})$ of pairwise distinct points of $\R^k$ and  $\Delta$ maps such a configuration to 
$(z_1, z_2, \dots, z_m, z_{m+1}, \dots,  z_{n+m}, z_{m+1}, \dots,  z_{n+m})$ (where the last $n$ points are repeated twice. Our statement now follows from the explicit description of the classes $\omega_{ij}$ and $\omega'_{ij}$ given in the proof of Proposition \ref{prop:HEBE}. 
\end{proof}
%

To prepare the tools to complete the proof of Theorem \ref{thm:main} we shall need Lemma \ref{lm:64} below. 
But first a few notations. Let $T\subset \{1, 2, \dots, m\}$ be a subset and let $p>m$. Consider the following cohomology classes in $H^\ast(E\times_B E)$:
$$\Omega_T = \prod_{i\in T} \omega_{i p}  \quad \mbox{and}
\quad \Omega'_T = \prod_{i\in T} \omega'_{i p}.$$
The classes $\Omega_T$ and $\Omega'_T$ are not among the generators of Proposition \ref{prop:63} but they can be explicitly expressed in terms of the generators, see below. 

From now on we shall assume that the dimension $k\ge 3$ is odd; the degree of the cohomology classes $\omega_{ij}$ and $\omega'_{ij}$ is then even and hence they commute. 
Besides, we shall formally introduce classes $\omega_{ij}$ with $i>j$ based on the convention 
$$\omega_{ij}=-\omega_{ji}.$$

\begin{lemma}\label{lm:64}
One has the identities
\begin{eqnarray}\label{identity}
\Omega_T = (-1)^{|T|-1} \cdot \sum_{i\in T} \left(\prod_{j\in T-\{i\}} \omega_{ij}\right)\cdot \omega_{ip}
\end{eqnarray}
and 
\begin{eqnarray}\label{identity1}
\Omega'_T = (-1)^{|T|-1} \cdot \sum_{i\in T} \left(\prod_{j\in T-\{i\}} \omega_{ij}\right)\cdot \omega'_{ip}.
\end{eqnarray}
\end{lemma}
\begin{proof}
The formula is obviously true for $|T|=1$ and we shall assume by induction that it is true for all subsets of cardinality 
smaller than $|T|$. Let $r$ be the maximal element of $T$ and let $T'=T-\{r\}$. Applying our induction hypothesis and the three term relation we get
\begin{eqnarray*}
\Omega_T &=& \Omega_{T'}\cdot \omega_{rp}  \\
&=&(-1)^{|T|} \sum_{i\in T'} \left(\prod_{j\in T'-\{i\}}\omega_{ij}\right)\cdot \omega_{ip}\cdot \omega_{rp}\\
&=& (-1)^{|T|} \sum_{i\in T'} \left(\prod_{j\in T'-\{i\}}\omega_{ij}\right)\cdot \omega_{ir}(\omega_{rp}-\omega_{ip})\\
&=& (-1)^{|T|-1} \sum_{i\in T'} \left(\prod_{j\in T-\{i\}}\omega_{ij}\right)\cdot\omega_{ip}
+ \left[(-1)^{|T|} \sum_{i\in T'} \left(\prod_{j\in T-\{i\}}\omega_{ij}\right)\right]\cdot \omega_{rp}.
\end{eqnarray*}

Applying the induction hypothesis again we get
$$\prod_{s\in T'} \omega_{sr} = (-1)^{|T|} \, \sum_{i\in T'}\left(\prod_{j\in T'-\{i\}} \omega_{ij}\right) \cdot \omega_{ir}
= (-1)^{|T|} \, \sum_{i\in T'}\left(\prod_{j\in T-\{i\}} \omega_{ij}\right).$$
Hence the sum in the square brackets equals
$$\prod_{s\in T'} \omega_{sr} = (-1)^{|T|-1} \prod_{j\in T-\{r\}} \omega_{rj}.$$
This completes the proof. 
\end{proof}

%
%

Consider the product
\begin{eqnarray}\label{x}
x\, =\, \prod_{i=2}^m (\omega_{i (m+1)}- \omega'_{i (m+1)})\cdot \prod_{j=m+1}^{m+n}(\omega_{1j}-\omega'_{1j})^2. 
\end{eqnarray}
It is a product of $2n+m-1$ elements of the kernel of the homomorphism induced by the diagonal map, see Proposition \ref{prop:64}. Showing $x\not=0$ will allow us to use Proposition \ref{cup} 
 to obtain the lower bound 
$$\tc[F(\R^k, n+m)\to F(\R^k, m)]\ge 2n+m-1$$ on the parametrised topological complexity of the Fadell-Neuwirth bundle. As this coincides with the previously obtained upper bound (\ref{eq:rough upper}), Theorem \ref{thm:main} would follow. 

Since we assume that $k\ge 3$ is odd we have 
$(\omega_{ij}-\omega'_{ij})^2 = - 2\omega_{ij} \omega'_{ij}$ (using Proposition \ref{prop:HEBE}). Hence $x\not=0$ is equivalent to $y\not=0$ where
\begin{eqnarray*}\label{y}
y\, &=&\, \prod_{i=2}^m (\omega_{i (m+1)}- \omega'_{i (m+1)})\cdot \prod_{j=m+1}^{m+n} (\omega_{1 j}\omega'_{1 j}) \\
&=& \, \prod_{i=2}^m (\omega_{i (m+1)}- \omega'_{i (m+1)})\cdot (\omega_{1 (m+1)}\omega'_{1(m+1)})\cdot  \prod_{j=m+2}^{m+n} (\omega_{1 j}\omega'_{1 j})\\
&=& \, \prod_{i=2}^m (\omega_{i (m+1)}- \omega'_{i (m+1)})\cdot (\omega_{1 (m+1)}\omega'_{1(m+1)})\cdot \omega_{IJ}\omega'_{IJ}.
\end{eqnarray*}
Here
$I=\{1, 1, \dots, 1\}$, and $J=\{m+2, m+3, \dots, m+n\}$; note that $I$ is a sequence of length $n-1$ consisting of $1$'s. 
Using  
notations introduced before Lemma \ref{lm:64} 
(with $p=m+1$)
we may write 
\begin{eqnarray}
y= \left[ \sum_{T, S} (-1)^{|S|}\Omega_T\cdot \Omega'_S
\right]\cdot \omega_{IJ}\omega'_{IJ}, 
\end{eqnarray}
where $T, S\subset \{1, 2, \dots, m\}$ run over all subset satisfying $T\cap S =\{1\}$ and $T\cup S=\{1, 2, \dots, m\}$. 
Using the result of Lemma \ref{lm:64} we obtain the following expression for the class $(-1)^{m+1}\cdot y$:
\begin{eqnarray*}
\sum_{T, S} (-1)^{|S|} \left\{ \sum_{i\in T} \left(\prod_{j\in T-i} \omega_{ij}\right)\cdot \omega_{i(m+1)}\right\}
\left\{ \sum_{\alpha\in S} \left(\prod_{\beta\in S-\alpha} \omega_{\alpha\beta}\right)\cdot \omega'_{\alpha(m+1)}\right\}\omega_{IJ}\omega'_{IJ}.
\end{eqnarray*}
Expanding the brackets gives a sum of many monomials, and each of these monomials is one of the generators described in Proposition \ref{prop:63}. However some monomials appear several times and may potentially cancel each other. 

Consider the monomial 
\begin{eqnarray}\label{monomial}
-\omega_{12}\omega_{23}\omega_{24}\dots\omega_{2 m}\omega_{1(m+1)}\omega'_{2 (m+1)}\omega_{IJ}\omega'_{IJ}
\end{eqnarray}
which arises by taking $T=\{1\}$, $S=\{1, 2, \dots, m\}$, $i=1$ and $\alpha=2$. It is easy to see 
that this is the only choice producing this monomial. Indeed, if the set $T$ contained an element $i>1$ then the factor 
$\omega_{1i}$ would be present which is not the case.  
If the set $S$ would miss a certain element $j>2$ then the factor $\omega_{2j}$ would not be present, however all such factors are present. Thus, we see that the class 
$y\in H^\ast(E\times_B E)$ is nonzero since it contain as a summand one of the free generators which cannot be cancelled by any other term. 

This completes the proof of Theorem \ref{thm:main}. \qed

\begin{ack}
Portions of this work were undertaken when the first and second authors visited the University of Florida Department of Mathematics in November, 2019. We thank the department for its hospitality and for providing a productive mathematical environment. We also thank the anonymous referees for their helpful comments.
\end{ack}

%

%
%
%

\newcommand{\arxiv}[1]{{\texttt{\href{http://arxiv.org/abs/#1}{{arXiv:#1}}}}}

\newcommand{\MRh}[1]{\href{http://www.ams.org/mathscinet-getitem?mr=#1}{MR#1}}

\bibliographystyle{amsplain}

\begin{thebibliography}{99}

%
%
%
%
\bibitem{BBK} Y. Baryshnikov, P. Bubenik, M. Kahle, \textit{Min-max Morse theory for configuration spaces of hard spheres}, Int. Math. Res. Not. IMRN 9(2014), 2577–2592; \MRh{3207377}

\bibitem{Bor} K. Borsuk, \textit{Theory of Retracts}, Monografie Matematyczne, Tom 44, Pa\'{n}stwowe Wydawnictwo Naukowe, Warsaw, 1967; \MRh{0216473}. 

\bibitem{Dee} K. Deeley, \textit{Configuration spaces of thick particles on a metric graph}, 
Algebraic \& Geometric Topology 11(2011), 1861–1892; \MRh{2826926}.


\bibitem{Dr} A. Dranishnikov, and R. Sadykov, \textit{
The topological complexity of the free product.}
Math. Z. 293 (2019), 407–416; \MRh{4002282}

%
%
%
%

\bibitem{E} R. Engelking, \textit{General Topology}, translated from the Polish by the author, second edition, Sigma Ser. Pure Math., 6, Heldermann, Berlin, 1989; \MRh{1039321}.

\bibitem{FN} 
E. Fadell and L. Neuwirth,
\textit{Configuration spaces}. 
Math. Scand. \textbf{10} (1962), 111-118; \MRh{0141126}

\bibitem{FH} E. Fadell and S. Husseini, \emph{Geometry and Topology of Configuration Spaces}, Springer Monographs in Mathematics, Springer-Verlag, Berlin, 2001; \MRh{1802664}.



\bibitem{Fa03} M. Farber, {\em Topological complexity of motion planning}, 
Discrete Comput. Geom. 29 (2003), no. 2, 211–221; \MRh{1957228}

\bibitem{Fa05} M. Farber,  
{\em{Topology of robot motion planning}}, 
in: \emph{Morse Theoretic Methods in Non-linear 
Analysis and in Symplectic Topology}, pp. 185--230, 
NATO Science Series II: Mathematics, Physics and Chemistry, vol. 217, 
Springer, 2006; 
\MRh{2276952}.

\bibitem{FG} M. Farber, M. Grant, 
\emph{Topological complexity of configuration spaces}, Proc. Amer. Math. Soc. \textbf{137} (2009), 1841--1847; \MRh{2470845}.

\bibitem{FGY} M. Farber, M. Grant, S. Yuzvinsky, 
\emph{Topological complexity of collision free motion planning algorithms in the presence of multiple moving obstacles}, in: \emph{Topology and Robotics}, pp. 75--83, Contemp. Math., vol. 438, Amer. Math. Soc., Providence, RI, 2007; \MRh{2359030}.

\bibitem{FY} M.~Farber, S.~Yuzvinsky,
{\em{Topological robotics: Subspace arrangements and collision
free motion planning}}, in: {\em Geometry, Topology, and 
Mathematical physics}, pp.~145--156, 
Amer. Math. Soc. Transl. Ser. 2, vol.~212, 
Amer. Math. Soc., Providence, RI, 2004; 
\MRh{2070052}.

\bibitem{GC} J.M. Garc\'ia-Calcines, \textit{A note on covers defining relative and sectional categories}, 
Topology Appl. {\bf 265} (2019); \MRh{3982636}

\bibitem{GLV} M. Grant, G. Lupton, L. Vandembrouq, eds., 
\emph{Topological Complexity and Related topics}, 
Contemp. Math., vol. 702. Amer. Math. Soc., Providence, RI, 2018; \MRh{3762828}.


\bibitem{Gonz} J. González, B. Gutiérrez, \textit{ Topological complexity of collision-free multi-tasking motion planning on orientable surfaces}. in: \textit{Topological Complexity and Related topics}, 151–163, Contemp. Math., 702, Amer. Math. Soc., Providence, RI, 2018; \MRh{3762838}

\bibitem{Hatcher} A.~Hatcher,
\emph{Algebraic Topology}, Cambridge University Press,
Cambridge, 2002; \MRh{1867354}

\bibitem{Lat} J.-C. Latombe, \emph{Robot Motion Planning}, Springer Science$+$Business Media, New York, 2012. 


\bibitem{Lav} S. M. LaValle, \emph{Planning Algorithms}, Cambridge University Press, Cambridge, 2006; 
\MRh{2424564}.




%
%
%

\bibitem{Sch} A.S. Schwarz, \textit{The genus of a fiber space}, Amer. Math. Sci. Transl., \textbf{55} (1966), 49–140; \MRh{0154284}.

\bibitem{Spa} E. Spanier, \emph{Algebraic Topology}, corrected reprint, Springer-Verlag, New York-Berlin, 1981; \MRh{0666554}. 

\bibitem{Sri}
T. Srinivasan, \emph{The Lusternik-Schnirelmann category of metric spaces},
Topology Appl. \textbf{167} (2014), 87--95; \MRh{3193429}. 


\end{thebibliography}

\end{document}